\documentclass[12pt]{amsart}

\usepackage{amsmath,textcomp,mathrsfs,textcomp,latexsym}

\usepackage{xcolor}
\usepackage{soul}

\definecolor{yel}{rgb}{1,.9,0}
\sethlcolor{yel}

\usepackage{color}
\definecolor{red-}{rgb}{1.0,0.0,0.0}
\definecolor{green-}{rgb}{0.0,0.7,0.0}
\definecolor{brown-}{rgb}{0.9,0.6,0.0}

\usepackage{amsfonts}
\usepackage{amssymb}
\usepackage[mathscr]{euscript}
\usepackage{amsmath}
\usepackage{amsthm}
\usepackage[all]{xy}
\usepackage{hyperref}
\hypersetup{
    colorlinks=true,
    linkcolor=blue,
    citecolor=blue,
    filecolor=blue,
    urlcolor=blue
}

\newtheorem{teo}{Theorem}
\newtheorem{prop}{Proposition}
\newtheorem{cor}{Corollary}
\newtheorem{lem}{Lemma}
\theoremstyle{definition}
\newtheorem{defi}{Definition}
\newtheorem{obs}{Remark}
\newtheorem{ex}{Example}
\newtheorem{prog}{Program}

\def\Q{\mathbb{Q}}

\def\F{\mathbb{F}}
\def\P{\mathbb{P}}

\usepackage{multicol}

\usepackage{amscd}
\usepackage{enumerate}

\usepackage{graphicx}

\usepackage{multirow}

\usepackage{verbatim}

\usepackage{tikz}

\begin{document}

\title{Lower and upper bounds for some generalized arcs}

\author{Alexis E. Almendras Valdebenito and Andrea Luigi Tironi}

\date{\today}

\address{{\small Departamento de Ciencias B\'asicas,
Universidad de Concepci\'on,
Los \'Angeles, Chile}}
\email{alexisalmendras@udec.cl}

\address{{\small Departamento de Matem\'atica,
Universidad de Concepci\'on,
Casilla 160-C,
Concepci\'on, Chile}}
\email{atironi@udec.cl}


\maketitle

\begin{abstract}
Let $\mathbb{F}_q$ be a field with $q$ elements. In this note, we study some generalized arcs, 
that is, sets of $\mathbb{F}_q$-points in the projective plane $\mathbb{P}^2(\mathbb{F}_q)$ 
such that no six of them are on a conic. First, we consider the geometric configurations of such generalized arcs
for small values of $q$ and then we give some upper and lower bounds for the cardinality of complete generalized arcs, 
i.e. generalized arcs which are not contained in a bigger~one.
\end{abstract}

\section*{Introduction}

Let $\mathbb{F}_q$ be a field with $q$ elements, where $q=p^r$ for some prime $p$ and $r\in\mathbb{Z}_{\geq 1}$.
Consider the finite projective plane $\P^2(\F_q)$. A $k$-arc is a set $\mathcal{K}$ of $k$ $\mathbb{F}_q$-points (or, simply, \textit{points}) of $\P^2(\F_q)$ such that no three of them lie on a $\mathbb{F}_q$-line (or, simply, \textit{line}) of $\P^2(\F_q)$.
The literature on $k$-arcs is very vast and many generalizations inherent to this concept have been considered by various authors, even in higher dimensions. In particular, recalling that a $k$-arc is complete if it is not contained in a $(k+1)$-arc, many authors focused their attention to the study of complete $k$-arcs and on some upper and lower bounds of their cardinality, obtaining interesting algebraic and geometric results, also in connection with some linear codes of special type. For a survey about all these results we refer, for instance, to \cite{HT} and the references therein.

In this note, we extend the concept of a $k$-arc along another direction. More precisely, we consider principally a
\textit{generalized $k$-arc}, that is, a set of $k$ points of $\P^2(\F_q)$ such that no six of them lie on a $\mathbb{F}_q$-conic (or, simply, \textit{conic}) of $\P^2(\F_q)$. Let us observe here that the conics we consider in this definition could be also reducible in two lines. In particular,  
a generalized $k$-arc which is also a $k$-arc is known simply as a \textit{Veronesian $k$-arc}. In line with the classical results about $k$-arcs, we start to study here some geometric properties of generalized $k$-arcs for small values of $q$ (see, e.g., Table \ref{arco:generalizado}) and we give some upper and lower bounds about the cardinality of complete generalized and Veronesian $k$-arcs for sufficiently bigger values of $q$ (see, e.g., Theorems \ref{teo322}, \ref{teo2}, \ref{t:2}). 

\medskip

The paper is organized as follows. After some basic notions and remarks in Section \ref{Sec 3.1}, we study in 
Section \ref{Sec 3.2} the cardinality and the geometry of maximal complete generalized $k$-arcs in $\P^2(\F_q)$ for $q\leq 9$ (Table \ref{arco:generalizado}). Subsequently, for higher values of $q$, we show some upper bounds for
generalized $k$-arcs, depending on the parity of $q$ (Theorem \ref{teo322} and Corollary \ref{corol}). In Section \ref{Sec 3.3}, we prove another lower bound for complete $k$-arcs in $\P^2(\F_q)$ (Proposition \ref{thm}) and we give principally a lower bound for complete generalized $k$-arcs $\mathcal{K}_g$, depending on the fact that $\mathcal{K}_g$ is 
complete as a Veronesian $k$-arc, or not (Theorems \ref{teo2} and \ref{t:2}). Finally, in the Appendix, we report the main Magma \cite{Magma} programs used to build all the tables and the corresponding examples.

\medskip

\noindent {\bf Acknowledgements}. Both authors would like to thank
the referee for reading the paper carefully and for many useful suggestions
which 
improved the presentation of some results of the first draft.
During the preparation of this paper in the framework of the Project Anillo ACT 1415
PIA CONICYT, the authors were partially supported by Proyectos 
VRID N. 214.013.039-1.OIN and N. 219.015.023-INV, and the first author was also partially supported by CONICYT-PCHA/Mag\'ister
Nacional a\~no 2013 - Folio: 221320380.

\section{Basic definitions and background material}\label{Sec 3.1}

First of all, let us recall here some basic classical definitions. A $(k,m)_n$-arc is a set of $k$ points of $\P^n(\F_q)$ such that no  $m+1$ of them lie on a hyperplane. A $(k,m)_n$-arc is complete if it is not properly contained in a $(h,m)_n$-arc, for some $h>k$. In particular, a $(k,2)_2$-arc is simply a $k$-arc in $\P^2(\F_q)$ (see, e.g., \cite{SEGRE1bis}, \cite{SEGRE2}, \cite{HIRSCHFELD0}). 
Moreover, the size of the largest complete $k$-arc in $\P^2(\F_q)$ will be denoted by $m(2,q)$, the size of the second largest complete $k$-arc 
in $\P^2(\F_q)$ by $m'(2,q)$, and the size of the smallest one by $t(2,q)$. 

\smallskip

Now, let us introduce here the two main definitions of this note.

\begin{defi}\label{def311}
A \textit{generalized $k$-arc} (or, simply, a $k_g$-arc) is a set of $k$ points of $\P^2(\F_q)$ such that not six of them are on a (possibly reducible) conic. A $k_g$-arc is said to be \textit{complete} if it is not contained in a $(k+1)_g$-arc. Moreover, the size of the largest complete $k_g$-arc is denoted  by $m_g(2,q)$, while the size of the smallest one will be denoted by $t_g(2,q)$. Finally, a generalized $k$-arc is said to be \textit{maximal} if $k=m_g(2,q)$ and 
\textit{minimal} if $k=t_g(2,q)$.
\end{defi}

\begin{defi}\label{def2}
A \textit{Veronesian $k$-arc} (or, simply, a $k_v$-arc) is a $k$-arc in $\P^2(\F_q)$ such that not six of its points 
lie on a (necessarily irreducible)
conic. A $k_v$-arc is said to be \textit{complete} if it is not contained in a $(k+1)_v$-arc. Moreover, the size of the largest complete $k_v$-arc is denoted  by $m_v(2,q)$, while the size of the smallest one will be denoted by $t_v(2,q)$. 
Finally, a generalized $k$-arc is said to be \textit{maximal} if $k=m_v(2,q)$ and 
\textit{minimal} if $k=t_v(2,q)$.
\end{defi}

\begin{obs}\label{obs313}
From Definitions \ref{def311} and \ref{def2}, one can deduce that a $k_g$-arc is at most a $(k,3)_2$-arc and that a Veronesian $k$-arc is a $k_g$-arc with the additional property of being a $(k,2)_2$-arc.
Moreover, since a Veronesian arc is a special case of a generalized arc, we have
$$t_v(2,q)\leq t_g(2,q)\quad \mathrm{and}\quad m_v(2,q)\leq m_g(2,q) \ .$$
\end{obs}

Denote by $|W|$ the cardinality of a set $W\subseteq\P^2(\F_q)$. The following simple technical result will be useful in the sequel. 

\begin{lem}\label{prop321}
Let $\mathcal{K}_g\subseteq\P^2(\F_q)$ be a generalized $k$-arc. If $l_1$ is a line in $\P^2(\F_q)$ such that $|\mathcal{K}_g\cap l_1|=3$, then for every line $l$ in $\P^2(\F_q)$ we have
\[ | (\mathcal{K}_g\setminus l_1)\cap l | \leq 2,\]
i.e. $(\mathcal{K}_g\setminus l_1)$ is a (Veronesian) $k$-arc in $\P^2(\F_q)$. In particular, a generalized $k$-arc is at most the disjoint union of a (Veronesian) $(k-3)$-arc and $3$ collinear points.
\end{lem}
\begin{proof}
By Remark \ref{obs313}, we have $|\mathcal{K}_g\cap l|\leq 3$ for every line $l$ in $\P^2(\F_q)$. By contradiction, suppose there exists a line $l_2$ in $\P^2(\F_q)$ such that $| (\mathcal{K}_g\setminus l_1)\cap l_2 |=3$. Then, we get $(l_1\cap\mathcal{K}_g)\cap(l_2\cap\mathcal{K}_g)=\varnothing$. Hence, the conic $\Gamma:=l_1\cup l_2$ is such that $|\mathcal{K}_g\cap \Gamma|=6$, 
but this is impossible because
$\mathcal{K}_g$ is a generalized $k$-arc.
\end{proof}

\section{Some upper bounds for $m_g(2,q)$}\label{Sec 3.2}

Before to study the geometry of generalized $k$-arcs for small values of $q$ and some upper bounds for $m_g(2,q)$
when $q$ is sufficiently large, let us give here the following result.

\begin{prop}\label{proposition}
We have $$m_g(2,q)\leq m(5,q)$$ for any $q\geq 2$, where $m(5,q)$ is the largest size of a $(k,5)_5$-arc.
\end{prop}

\begin{proof}
Consider the Veronese embedding
\begin{equation}\label{Ver}
\begin{aligned}
\nu_2 \colon \P^2(\F_q) &\longrightarrow \P^5(\F_q) \\
	[x_0:x_1:x_2]&\longmapsto [x_0^2 : x_0x_1 : x_0x_2 : x_1^2 : x_1x_2 : x_2^2]\ .\\
\end{aligned}
\end{equation}
Define $S:=\nu_2(\P^2(\F_q))\subset \P^5(\F_q)$ and denote by $[z_0:z_1:z_2:z_3:z_4:z_5]$ the general point of $\P^5(\F_q)$. If $H$ is a hyperplane
$\sum_{i=0}^5 \alpha_i z_i =0$ in $\P^5(\F_q)$, with $\alpha_i\in\F_q$, then $\nu_2^{-1}(H \cap S)$ is a conic in $\P^2(\F_q)$. If $\mathcal{K}_g$ is a generalized $k$-arc, then  $|\nu_2(\mathcal{K}_g)\cap H |\leq 5$. Therefore, the set $\nu_2(\mathcal{K}_g)$ is a $(k,5)_5$-arc and then $m_g(2,q)\leq m(5,q)$. 
\end{proof}

First of all, let us compute  the exact values of $m_g(2,q)$ for small  $q$'s, that is, for $q=2,3,4,5, 7$ and $8$.

\begin{prop}\label{prop325}
$m_g(2,2)=7$ $($cf. \em{Table \ref{arco:generalizado}}$)$.
\end{prop}

\begin{proof}
Note that any conic in $\P^2 (\F_2)$ has at most $5$ points. So, any subset of $\P^2 (\F_2)$ is  a generalized arc. Therefore, $m_g(2,2)=|\P^2(\F_2)|=7$.
\end{proof}

\begin{ex}\label{ex2}
Let $q=3$ and let $\mathcal{K}=\{p_1,p_2,p_3,p_4\}$ be a complete $4$-arc in $\P^2(\F_3)$. Denote by $l_{ij}$ the line passing through the distinct points $p_i$ and $p_j$. Let $q_1$ be the intersection of $l_{12}$ and $l_{34}$, and let $q_2$ be the intersection of $l_{13}$ and $l_{24}$. Consider the point  $q_3$ obtained by intersecting the line  passing through $q_1$ and $q_2$ with $l_{14}$. So, by construction, the set $\mathcal{K}_g=\{ p_1,p_2,p_3,p_4,q_1,q_2,q_3\}$ is a generalized $7$-arc. For instance, consider the set $\mathcal{K}_g$ given by the following seven points (see Figure \ref{arcoge}):
\begin{multicols}{4}
\begin{itemize}
\item $[0:1:2]$
\item $[1:1:1]$
\item $[1:1:2]$
\item $[1:0:1]$
\item $[1:2:0]$
\item $[1:1:0]$
\item $[0:1:0]$
\end{itemize}
\end{multicols}

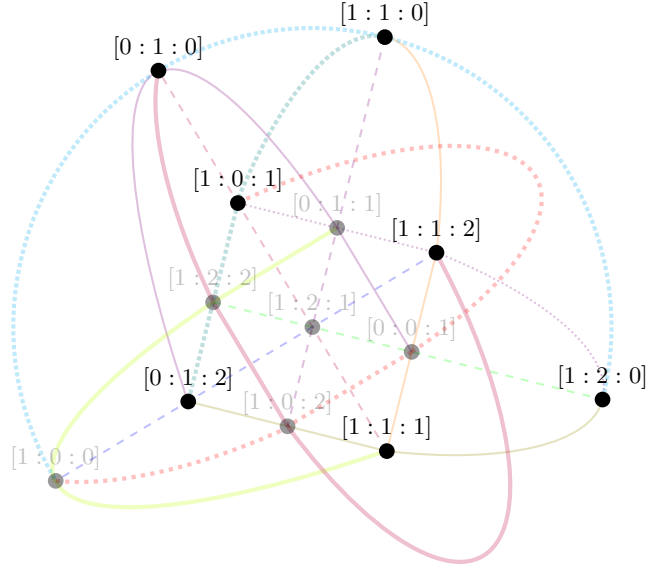
\begin{figure}[!h]
\begin{center}
\begin{tikzpicture}[scale=0.33]
\path (0,0) coordinate (A);
\path (8,-2) coordinate (B);
\path (10,6) coordinate (C);
\path (2,8) coordinate (D);
\path (1,4) coordinate (E);
\path (6,7) coordinate (F);
\path (9,2) coordinate (G);
\path (4,-1) coordinate (H);
\path (-5.34,-3.2) coordinate (I);
\path (-1.2,13.34) coordinate (J);
\path (7.92,14.69) coordinate (K);
\path (16.69,0.08) coordinate (L);
\path (5,3) coordinate (M);
\draw[opacity=0.25,thick,olive] (A)--(B);
\draw[opacity=0.25,thick, orange] (B)--(C);
\draw[opacity=0.25,thick,violet,densely dotted] (C)--(D);
\draw[opacity=0.25,ultra thick,teal,densely dotted] (D)--(A);
\draw[opacity=0.25,thick,violet] (F)--(G);
\draw[opacity=0.25,red, dotted,ultra  thick] (G)--(H);
\draw[opacity=0.25,ultra thick,purple] (H)--(E);
\draw[opacity=0.25,ultra thick,lime] (E)--(F);
\draw[opacity=0.25,dashed,green,thick] (E)--(G)--(L);
\draw[opacity=0.25,dashed,blue,thick] (C)--(I);
\draw[opacity=0.25,dashed,purple,thick] (B)--(J);
\draw[opacity=0.25,dashed,violet,thick] (K)--(H);
\draw[opacity=0.25,rotate=-14.04,thick,olive] (B) arc [start angle=-70,end angle=0, x radius=12.05, y radius= 4.39];
\draw[opacity=0.25,rotate=-14.04,thick,violet,densely dotted] (L) arc [start angle=0,end angle=70, x radius=12.05, y radius= 4.39];
\draw[opacity=0.25,rotate=75.6,thick,orange] (C) arc [start angle=-70,end angle=0, x radius=12.05, y radius= 4.39];
\draw[opacity=0.25,rotate=75.6,ultra thick,teal,densely dotted] (K) arc [start angle=-0,end angle=70, x radius=12.05, y radius= 4.39];
\draw[opacity=0.25,rotate=114.4,thick,violet] (F) arc [start angle=-54,end angle=0, x radius=21.05, y radius= 4.83];
\draw[opacity=0.25,rotate=114.4,thick,violet] (J) arc [start angle=0,end angle=67, x radius=21.05, y radius= 4.83];
\draw[opacity=0.25,rotate=24.53,ultra thick,lime] (E) arc [start angle=124,end angle=180, x radius=19.97, y radius= 4.76];
\draw[opacity=0.25,rotate=24.53,ultra thick,lime] (I) arc [start angle=180,end angle=248, x radius=19.97, y radius= 4.76];
\draw[opacity=0.25,rotate=119.67,ultra thick,purple] (J) arc [start angle=22,end angle=80, x radius=12.21, y radius= 4.44];
\draw[opacity=0.25,rotate=119.67,ultra thick,purple] (H) arc [start angle=108,end angle=268, x radius=12.21, y radius= 4.44];
\draw[opacity=0.25,red, dotted,rotate=115.38,ultra thick] (I) arc [start angle=123,end angle=178, x radius=4.54, y radius= 11.47];
\draw[opacity=0.25,red, dotted,rotate=115.38,ultra thick] (G) arc [start angle=200,end angle=360, x radius=4.35, y radius= 11.47];
\draw[opacity=0.25,cyan, ultra thick,densely dotted] (L) arc [start angle=-14.04,end angle=210.96,radius=12.05];
\fill (A) circle [radius=9pt] node [above] {\scriptsize $[0:1:2]$}; 
\fill (B) circle [radius=9pt] node [above] {\scriptsize $[1:1:1]$};
\fill (C) circle [radius=9pt] node [above] {\scriptsize $[1:1:2]$};
\fill (D) circle [radius=9pt] node [above] {\scriptsize $[1:0:1]$} ;
\fill[opacity=0.4] (E) circle [radius=9pt] node [above,opacity=0.25] {\scriptsize $[1:2:2]$};
\fill[opacity=0.4] (F) circle [radius=9pt] node [above,opacity=0.25] {\scriptsize $[0:1:1]$};
\fill[opacity=0.4] (G) circle [radius=9pt] node [above,opacity=0.25] {\scriptsize $[0:0:1]$};
\fill[opacity=0.4] (H) circle [radius=9pt] node [above,opacity=0.25] {\scriptsize $[1:0:2]$};
\fill[opacity=0.4] (I) circle [radius=9pt] node [above,opacity=0.25] {\scriptsize $[1:0:0]$};
\fill (J) circle [radius=9pt] node [above] {\scriptsize $[0:1:0]$};
\fill (K) circle [radius=9pt] node [above] {\scriptsize $[1:1:0]$};
\fill (L) circle [radius=9pt] node [above] {\scriptsize $[1:2:0]$};
\fill[opacity=0.4] (M) circle [radius=9pt] node [above,opacity=0.25] {\scriptsize $[1:2:1]$};
\end{tikzpicture}
\end{center}
\caption{A maximal generalized $7$-arc in $\P^2(\F_3)$.}
\label{arcoge}
\end{figure}
\end{ex}

\begin{prop}
$m_g(2,3)=7$ $($see {\em Figure \ref{arcoge}}; cf. \em{Table \ref{arco:generalizado}}$)$.
\end{prop}

\begin{proof}
Note that any irreducible conic in $\P^2(\F_3)$ has four points. Hence, by Lemma \ref{prop321} we deduce that $m_g(2,3)\leq m(2,3)+3=7$ and we can conclude by Example~\ref{ex2}. 
\end{proof}

Consider now, from a combinatorial and geometric point of view, the more intricate cases $q=4,5$. 
For $q=4$, we have the following result.

\begin{prop}\label{prop4}
$m_g(2,4)=7$ $($cf. \em{Table \ref{arco:generalizado}}$)$.
\end{prop}

\begin{proof}
First of all, we prove a lower bound for $m_g(2,4)$.

\medskip 

\noindent Claim. $m_g(2,4)\geq 7$.

\smallskip

\noindent Since $q=4$, consider a complete $6$-arc $\mathcal{K}\subset \P^2(\F_4)$, i.e. a conic and its nucleus (that is, the unique intersection point of the $q+1$ tangent $\mathbb{F}_q$-lines to an irreducible conic in $\P^2(\F_q)$ with $q$ even \cite[p. 143]{HIRSCHFELD0}). Let $p\notin\mathcal{K}$ and write $\mathcal{K}_g:=\mathcal{K}\cup \{p\}$. Suppose that $\mathcal{K}_g$ is not a generalized $7$-arc. As any irreducible conic has five points, it follows that there exist two lines $l_1$ and $l_2$ such that $(l_1\cap \mathcal{K}_g)\cap (l_2\cap \mathcal{K}_g)=\varnothing$  and  $|\mathcal{K}_g\cap l_i | =3$, for $i=1,2$. Since $|\mathcal{K}|=6$, after renaming, we can assume that $p\in l_1\setminus l_2$. Hence $|\mathcal{K}\cap l_2|=3$, but this is a contradiction because $\mathcal{K}$ is a $6$-arc. Thus, $\mathcal{K}_g$ is a generalized $7$-arc and $m_g(2,4)\geq 7$. \hfill{Q.E.D.}

\medskip

\noindent By the Claim, we know that there exists a generalized $7$-arc in $\P^2(\F_4)$. Now, assume that there exists a generalized $8$-arc $\mathcal{K}_g\subset \P^2(\F_4)$. 
Note that $\mathcal{K}_g$ cannot be a $8$-arc because 
otherwise we would have $|\mathcal{K}_g|\leq q+2=6$ (see e.g. \cite[Chapter $8$]{HIRSCHFELD0}). Thus, there is a line $l\subset \P^2(\F_4)$ such that $|\mathcal{K}_g\cap l|=3$. By Lemma \ref{prop321}, we know that $\mathcal{K}_g\setminus l=:\mathcal{K}$ is a $5$-arc in $\P^2(\F_4)$. 
Write $\mathcal{K}_g:=\{p_1,\dots,p_8\}$ and, after renaming, suppose that $p_1,\dots,p_5\in\mathcal{K}$ and $p_6,p_7,p_8\in l$. Let $p$ be the nucleus of the irreducible conic $\gamma$ containing $\mathcal{K}$. Since $|\gamma|=|\mathcal{K}|=q+1=5$, note that $|\mathcal{K}\cap l|=0$, otherwise $l$ would contain four distinct points of $\mathcal{K}_g$ and taking any other secant line $l'$ of $\mathcal{K}_g$ such that $(l\cap \mathcal{K}_g)\cap (l'\cap \mathcal{K}_g)=\varnothing$, the reducible conic $l\cup l'$ would have $6$ points in common with $\mathcal{K}_g$, a contradiction because $\mathcal{K}_g$ is a generalized $8$-arc. In particular, this shows that $p\notin l$ by definition of the nucleus $p$. Let $q_1$ and $q_2$ be the other two points of $l$ distinct from $p_6,p_7$ and $p_8$. Considering the pencil of $\mathbb{F}_4$-lines through $p$, after renaming, we can assume that the line $l_{16}:=\langle p_1,p \rangle$ is such that $p_6\in l_{16}$. Furthermore, considering the pencil of $\mathbb{F}_4$-lines through $p_1$, after renaming, we can suppose that there exist three lines $l_{1i}:=\langle p_1,p_i \rangle$ through $p_1$ and $p_i$ for $i=2,4,5$ such that $q_2\in l_{12}, p_7\in l_{15}$ and $p_8\in l_{14}$. Now, note that $p_6\notin l_{12}$ and that the tangent line $l'$ to $\gamma$ at $p_2$ cannot contain the point $p_6$ because $p\in l'$ and $l_{16}=\langle p, p_6\rangle$. So, consider the pencil of $\mathbb{F}_4$-lines through $p_2$. Write $l_{2j}:=\langle p_2,p_j\rangle$ for $j=3,4,5$. Since $p_6\notin l'\cup l_{12}$, we see that there exists $j\in\left\{3,4,5\right\}$ such that $p_6\in l_{2j}$. Thus one of the following three cases can occur: (a) $j=3, \ \left\{p_2,p_3,p_6\right\}\subset l_{23}$, \ (b) $j=4, \ \left\{p_2,p_4,p_6\right\}\subset l_{24}$, \ (c) $j=5, \ \left\{p_2,p_5,p_6\right\}\subset l_{25}$. In all these cases, having in mind that $\left\{p_1,p_4,p_8\right\}\subset l_{14}$ and $\left\{p_1,p_5,p_7\right\}\subset l_{15}$, we get a contradiction by taking the reducible conics $l_{23}\cup l_{14}, l_{24}\cup l_{15}, l_{25}\cup l_{14}$, respectively. This shows that actually $m_g(2,4)=7$.
\end{proof}

\begin{ex}\label{ex2bis}
Let $\mathbb{F}_4=\{0,1,a,a+1\}$ be a field with four elements, where $a^2+a+1=0$.
Considering in $\mathbb{P}^2(\mathbb{F}_4)$ the conic $x_0^2+x_1x_2=0$, its nucleus and the point $[1:0:1]$, from the Claim of Proposition \ref{prop4} it follows that the set 
$$\mathcal{K}_g:=\{ [1:1:1], [a:a+1:1], [a+1:a:1], [0:0:1], [0:1:0], [1:0:0], [1:0:1] \} $$
is a maximal generalized $7$-arc in $\P^2(\F_4)$ (cf. Table \ref{arco:generalizado}).
\end{ex}

Now, consider the case $q=5$. 

\begin{ex}\label{ex3} 
In $\P^2(\F_5)$, consider the following set
\[ \mathcal{K}':=\{ [1:0:0], [0:1:0], [0:0:1], [1:1:1], [1:0:1], [1:1:0], [0:1:1] \}. \]
By the Veronese embedding $\nu_2\colon \P^2(\F_5)\to \P^5(\F_5)$, we have
\begin{multline*}
 \nu_2(\mathcal{K}')=\{ [1:0:0:0:0:0],
[0:1:0:0:0:0],[0:0:1:0:0:0],\\ [1:1:1:1:1:1], [1:0:1:0:1: 0],[1:1:0:1:0:0],[0:1:1:0:0:1] \},
\end{multline*}
 and using the homography given by the matrix 
\[ \begin{pmatrix}
1 &0& 0 &4 &4 &1\\
0& 1 &0& 4 &0 &0\\
0 &0 &1& 0 &4 &0\\
0& 0 &0& 0& 0& 1\\
0& 0 &0 &0 &4 &1\\
0 &0 &0& 4& 0& 1\\
\end{pmatrix},
\]
we see that $\nu_2(\mathcal{K}')$ is projectively equivalent to the $(7,5)_5$-arc 
$$\mathcal{K}:=\left\{ [1:0:\dots:0],[0:1:0:\dots:0],\dots,[0:\dots:0:1],[1:1:\dots:1]\right\}\ .$$ 
So, $\mathcal{K}'$ is a generalized $7$-arc in $\P^2 (\F_5)$ (cf. Table \ref{arco:generalizado}).
\end{ex}

\medskip

\begin{prop}\label{prop5}
$m_g(2,5)=7$ $($cf. \em{Table \ref{arco:generalizado}}$)$.
\end{prop}

\begin{proof}
By Proposition \ref{proposition}, we have $m_g(2,5)\leq m(5,5)=7$. Thus, we can conclude by Example \ref{ex3}.
\end{proof}

For $q=7,8$, from Proposition \ref{proposition} and \cite[Theorem 3.17 (iv) and (xiii)]{HIRSCHFELD}, it follows that $m_g(2,q)\leq m(5,q)=q+1$. Using Program \ref{prog1}, we can see that in fact $m_g(2,7)=8$ and $m_g(2,8)=9$ (see Table \ref{arco:generalizado}). In general, this program could theoretically find the exact values of $m_g(2,q)$, but already for  
$q\geq 9$ the geometry of a generalized $k$-arc becomes much more intricate and, for this reason,
we will give only some upper bounds for $m_g(2,q)$. On the other hand, when $q=9$, some geometric considerations, up to projectivities, together with slight modifications of Program \ref{prog1} allowed us to obtain $m_g(2,9)=8$ and 
an example of a complete generalized $8$-arc is given in Table \ref{arco:generalizado}.

\begin{table}[!h]
\begin{center}\renewcommand{\arraystretch}{1.6}
\setlength{\tabcolsep}{10pt}
\begin{tabular}{ | c | c c  c | p{8cm} | }\hline
~~$q$~~  & $m'(2,q)$ & $m(2,q)$ & $m_g(2,q)$ & Examples of complete generalized $k$-arcs  \\ \hline \hline
2 &  4& 4 & 7 & $\P^2(\F_2)$ \\ \hline
3 & 4 & $4$ & 7 &  $[0:1:2], [1:1:1], [1:1:2], [1:0:1],$\newline $ [1:2:0], [1:1:0], [0:1:0]$\\ \hline
4 & 6 & 6 & 7   & $[1:0:0], [0:1:0], [0:0:1], [1:1:1],$\newline     $[\alpha:\alpha^2:1],[\alpha^2:\alpha:1],[1:0:1]$ \\ \hline
5 & 6 & 6 & 7 &  $[1:0:0], [0:1:0], [0:0:1], [1:1:1],$\newline $ [1:0:1], [1:1:0], [0:1:1]$ \\ \hline
7 & 6 & 8 & 8 & $[1:0:0], [0:1:0], [0:0:1], [1:1:1],$\newline $[2:2:1],[5:2:1],[3:5:1],[6:1:1]$  \\ \hline
8 & 6 & 10 & 9 & $[1:0:0],[0:1:0],[0:0:1],[1:1:1],$\newline $[1:\alpha^4:1],[\alpha:1:1],[\alpha^5:\alpha^6:1],$ \newline$[\alpha:\alpha^4:1],[\alpha^5:1:0]$ \\ \hline
9 &  8 & 10  &  8 & $[1:0:0],[0:1:0],[0:0:1], [1:1:1],$\newline $[\alpha^6: \alpha: 1], [ \alpha^5: 0 :1], [\alpha :\alpha: 1], [\alpha^7 :\alpha^2 :1]$ \\ \hline
\end{tabular}
\medskip
\caption{Maximal generalized $k$-arcs in $\mathbb{P}^2(\mathbb{F}_q)$ for $q\leq 9$, where $\mathbb{F}_q^*=\langle\alpha\rangle$.}
\label{arco:generalizado}
\end{center}
\end{table}

\medskip

For sufficiently big values of $q$, we obtain the following upper bounds for $m_g(2,q)$.

\begin{teo}\label{teo322}
We have
\[ m_g(2,q)\leq \begin{cases}
\min\left\{ m(5,q),\ m'(2,q)+3,\ q-\dfrac{\sqrt{q}}{4}+\dfrac{19}{4} \right\} & \ \ \mathrm{if}\ q\geq 7\ \mathrm{is\ odd} \\ \\
\smallskip
\min\left\{ m(5,q),\ q-\dfrac{\sqrt{q}}{2}+\dfrac{17}{4} \right\} & \ \ \mathrm{if}\ q\geq 16\ \mathrm{is\ even}\ , \\
\end{cases}
\]
where $m(5,q)$ is the largest size of a $(k,5)_5$-arc and
$m'(2,q)$ is the second largest size of a complete $(k,2)_2$-arc.
\end{teo}
 
\begin{proof}
First, suppose that $q\geq 7\ \mathrm{is\ odd}$.
From \cite[Theorem I]{SEGRE1},  we know that a complete $k$-arc with $m(2,q)$ points is a conic. Thus, if a $k$-arc $\mathcal{K}$ has more than $m'(2,q)$ points, then $\mathcal{K}$ can be extended to a conic $\Gamma$, so $\mathcal{K}$ must lie on $\Gamma$.
By Lemma \ref{prop321}, we can assume that a generalized $k$-arc is contained in the disjoint union of a $k$-arc and $3$ collinear points. Since $m'(2,q)\geq 6$ for any odd integer $q\geq 7$, we conclude that $m_g(2,q)\leq m'(2,q)+3$. 
Assume now that one of the following conditions is satisfied:
\begin{enumerate}
\item $q\geq 7\ \mathrm{is\ odd}$ and there is a generalized $k$-arc $\mathcal{K}_g$ with $k> q-\frac{\sqrt{q}}{4}+\frac{19}{4}$; 
\item $q\geq 16\ \mathrm{is\ even}$ and there exists a generalized $k$-arc $\mathcal{K}_g$ with $k> q-\frac{\sqrt{q}}{2}+\frac{17}{4}$.
\end{enumerate}
Consider the Veronese embedding $\nu_2 :\ \P^2(\F_q)\longrightarrow \P^5(\F_q)$ defined as in \eqref{Ver}.
Note that the set $\nu_2(\mathcal{K}_g)$ is a $(k,5)_5$-arc. By \cite[Theorems 4.10 and 4.11]{HT},
$\nu_2(\mathcal{K}_g)$ lies on a unique normal rational curve $\Gamma$ of degree $5$ in $\mathbb{P}^5(\mathbb{F}_q)$.
Moreover, $\nu_2(\mathcal{K}_g)\subseteq \nu_2(\mathbb{P}^2(\mathbb{F}_q))=:S$, where $S$ is the Veronese surface. Since $S$ is defined by quadrics (see \cite[Example 2.7]{HARRIS}), we see  that   $\Gamma\not\subseteq S$. Furthermore, there exists a quadric $\Q^4\subset\P^5(\F_q)$ such that $S\subset \Q^4$ and $\Gamma\not\subset \Q^4$. So, we get  
\[ | \Gamma \cap S | \leq | \Gamma \cap \Q^4 | \leq \text{deg}(\Gamma)\cdot  \text{deg}(\Q^4) =10, \]
where the last inequality is due to the B\'ezout's Theorem \cite[Theomem 25.1]{LEFSCHETZ}. This implies that
$k=| \nu_2(\mathcal{K}_g) |=| \nu_2(\mathcal{K}_g) \cap \Gamma \cap S | \leq | \Gamma \cap S |\leq 10,$
but this gives a numerical contradiction in both cases $(1)$ and $(2)$. 
We conclude by Proposition \ref{proposition}.
\end{proof}

\begin{obs}
For $q=2^h$ and $h\geq 4$, there exist  irregular hyperovals, that is,  $m(2,q)$-arcs in $\mathbb{P}^2(\mathbb{F}_q)$ which are not the union of a conic and its nucleus. For this reason, in these cases we cannot use the same argument as in the first part of the proof of 
 Theorem~\ref{teo322}.
\end{obs}

\begin{obs}\label{obs3210}
From \cite[$\S 4.5.2$]{HT} we know that 
\[ 
m'(2,q)\leq \begin{cases}
q-1 & \ \ q \geq 7 \quad \text{\hspace{2cm} Segre 1955, Tallini 1957 \cite{SEGRE1ter}}\\
q-\frac{1}{4}\sqrt{q}+\frac{25}{16} & \ \ q \ \text{odd \hspace{2.2cm} Thas 1987 \cite{THAS}} \\
\frac{44}{45}q+\frac{8}{9}  & \ \ q \ \text{prime \hspace{1.832cm} Voloch 1990 \cite{VOLOCH}} \\
q-\frac{1}{2}\sqrt{q}+5  &  \ \ q=p^h, p\geq 5 \text{\hspace{1.05cm} Hirschfeld-Korchm\'aros 1996 \cite{HK} .} \\
\end{cases}
\]
\end{obs}

\begin{obs}\label{obs3210bis}
For $q\geq 8$ even we have $m(5,q)=q+1$, while for $q\geq 7$ odd we know that $m(5,q)=q+1$, except possibly for $23\leq q\leq 83$, Hirschfeld 1997 \cite{HIRSCHFELD}.
\end{obs}

The next result can be easily obtained by comparing directly Theorem \ref{teo322} and the results recalled in Remarks \ref{obs3210} and \ref{obs3210bis}.

\begin{cor}\label{corol}
Let $q\geq 16$ even. Then
\[ m_g(2,q)\leq \begin{cases}
q+1& \ \ q=16, 32  \\
\left\lfloor q-\dfrac{\sqrt{q}}{2}+\dfrac{17}{4}\right\rfloor & \ \ q\geq 64\ . \\
\end{cases}
\]
Let $q\geq 7$ odd. Then we have one of the following cases:
\begin{itemize}
\item for $q$ prime,
\[ m_g(2,q)\leq \begin{cases}
q+1& \ \ 7\leq q \leq 19\\
q+2& \ \ 23\leq q \leq 83\\
\left\lfloor \dfrac{44}{45}q+\dfrac{8}{9} \right\rfloor +3 & \ \ q\geq 89 \ ; \\
\end{cases}
\]
\item for $q=p^s$, with $p$ prime greater than or equal to $5$ and $s\geq 2$:
\[ m_g(2,q)\leq \begin{cases}
q+2& \ \ q \in \left\{25,49\right\}\\
q+1& \ \ q \in \left\{121,125\right\}\\
\left\lfloor q-\dfrac{\sqrt{q}}{2}+5 \right\rfloor +3 & \ \ q\geq 169 \ ; \\
\end{cases}
\]
\item for $q=3^s$  and $s\geq 2$:
\[ m_g(2,q)\leq \begin{cases}
10& \ \ q=9\\
q+2& \ \ q \in \left\{27,81\right\}\\
\left\lfloor q-\dfrac{\sqrt{q}}{4}+\dfrac{25}{16} \right\rfloor +3 & \ \ q\geq 243 \ .\\
\end{cases}
\]
\end{itemize}
\end{cor}

\smallskip

\section{Some lower bounds for $t(2,q)$, $t_v(2,q)$ and $t_g(2,q)$}\label{Sec 3.3}

Recall that the size of a smallest complete $k$-arc in $\P^2(\F_q)$ is denoted by $t(2,q)$. In \cite{BALL}, S. Ball showed that 
\begin{equation}\label{*}
t(2,q)\geq \begin{cases}
\lfloor \sqrt{2q}+2 \rfloor &\quad \text{for any $q$} \\
\left\lceil \sqrt{3q}+\dfrac{1}{2} \right\rceil & \quad \text{for $q=p^h$, $p$ prime, $h=1,2$}\ .\\
\end{cases}\tag{*}
\end{equation}

In \cite{BDFMP}, the authors realized a more detailed work on lower bounds when $q\leq 7559$
and they proved that $t(2,q)\leq \lfloor 3 \sqrt{q} \rfloor$ for $q\leq 89$. 

\medskip

First of all, let us give here a lower bound for $t(2,q)$ which slightly improves $\lfloor \sqrt{2q}+2 \rfloor$ only for some values of $q$ 
(see the below Remark \ref{rem}).

\begin{prop}\label{thm}
$t(2,q)\geq \left\lceil \sqrt{\left( 2q+\frac{1}{4}\right) }+\frac{3}{2}\right\rceil $ for any $q\geq 2$.
\end{prop}

\begin{proof}
Assume that a set of $k$ points $\{ p_1,p_2,\dots,p_k\}$ is a complete $k$-arc. Denote by $l_{ij}$ the unique $\mathbb{F}_q$-line passing through $p_i$ and $p_j$, where $i\not= j$ and $i,j\in\{1,\dots,k\}$. Write
$$\mathbb{P}^2(\mathbb{F}_q)=\bigcup_{1\leq i<j\leq k}l_{ij}=\left(\bigcup_{h=1}^{k-1}l_{hk}\right)\cup\left(\bigcup_{1\leq i<j\leq  k-1}l_{ij}\setminus \left( \cup_{h=1}^{k-1}l_{hk} \right)\right)\ .$$
Then we deduce that
$$q^2+q+1=|\mathbb{P}^2(\mathbb{F}_q)|\leq |\cup_{h=1}^{k-1}l_{hk}|+\sum_{1\leq i<j\leq k-1}|l_{ij}\setminus \left( \cup_{h=1}^{k-1}l_{hk}\right)|=$$
$$=[(k-1)q+1]+\sum_{1\leq i<j\leq k-1}[q+1-(k-1)]=(k-1)q+1+\binom{k-1}{2}(q+2-k)\ ,$$
i.e. $q^2+q+1-[(k-1)q+1]-\binom{k-1}{2}(q+2-k)\leq 0$. Hence we get
$$(q+2-k)\cdot \left[ q-\binom{k-1}{2} \right]\leq 0\ , $$
that is,
$$\left[k-(q+2)\right]\cdot \left[k-\left(\frac{3}{2}-\sqrt{2q+\frac{1}{4}}\right)\right]\cdot \left[k-\left(\frac{3}{2}+\sqrt{2q+\frac{1}{4}}\right)\right]\leq 0\ .$$
This leads to the solution $\frac{3}{2}+\sqrt{2q+\frac{1}{4}}\leq k\leq q+2$ which shows that the inequality of the statement holds for any $q\geq 2$.
\end{proof}

\medskip

In Tables \ref{arco:generalizadoinferior} and \ref{arco:generalizadoinferior-bis}, we make a comparison between 
\eqref{*} and  the lower bound of Proposition \ref{thm} 
when $q\leq 31$.  
Furthermore, in Table \ref{arco:generalizadoinferior} we give also
some examples which reach the best known lower bound $t(2,q)$ for a $k$-arc when $q\leq 11$.  

{\tiny
\begin{table}[!h]
\begin{center}\renewcommand{\arraystretch}{1.6}
\setlength{\tabcolsep}{6pt}
\begin{tabular}{ | c | c c  c | c | p{4.1cm} | }\hline
$~~q~~$ & $\lfloor \sqrt{2q} +2 \rfloor$ & $\lceil{ \sqrt{3q}+\frac{1}{2}  }\rceil$ & $\lceil \sqrt{\left( 2q+\frac{1}{4}\right) }+\frac{3}{2}\rceil$ & $t(2,q)$ & Examples of arcs in $\P^2(\F_q)$\\ \hline \hline
2&\textbf{4}&3&\textbf{4}&\textbf{4} &  $[0:0:1],[0:1:0],[1:0:0],$\newline $[1:1:1]$      \\ \hline
3&\textbf{4}&\textbf{4}&\textbf{4}&\textbf{4} & $[0:0:1],[0:1:0],[1:0:0],$\newline $[1:1:1]$   \\ \hline
4&4&4&5& 6 & $[0:0:1],[0:1:0],[\alpha:\alpha^2:1],$\newline $[1:1:1],[1:0:0],[\alpha^2:\alpha:1]$   \\ \hline
5&5&5&5&6 & $[0:0:1],[0:1:0],[1:0:0],$\newline $[1:1:1],[3:2:1],[4:3:1]$   \\ \hline
7&5&\textbf{6}&\textbf{6}&\textbf{6} & $[0:0:1],[0:1:0],[1:0:0],$\newline $[1:1:1],[3:2:1],[4:3:1]$   \\ \hline
8&\textbf{6}&-&\textbf{6}&\textbf{6} & $[0:0:1],[0:1:0],[\alpha^2:\alpha^3:1],$ \newline $[1:0:0],[1:1:1],[\alpha^4:\alpha^2:1]$   \\ \hline
9&\textbf{6}&\textbf{6}&\textbf{6}&\textbf{6} & $[0:0:1],[0:1:0],[\alpha^2:\alpha^5:1],$ \newline $[1:0:0],[1:1:1],[\alpha^5:\alpha^2:1]$   \\ \hline
11&6&\textbf{7}&\textbf{7}&\textbf{7}& $[0:0:1],[0:1:0],[0:0:1],$ \newline $[1:1:1],[7:3:1],[3:2:1],$\newline $[10:8:1]$    \\ \hline
\end{tabular}
\medskip
\caption{{\small Comparison between the known lower bounds for arcs in $\P^2(\F_q)$ for $2\leq q\leq 11$, where $\mathbb{F}_q^*=\langle\alpha\rangle$.} }
\label{arco:generalizadoinferior}
\end{center}
\end{table}
}

{\tiny
\begin{table}[!h]
\begin{center}\renewcommand{\arraystretch}{1.6}
\setlength{\tabcolsep}{6pt}
\begin{tabular}{ | c | c c  c | c | p{5.3cm} }\hline
$~~q~~$ & $\lfloor \sqrt{2q} +2 \rfloor$ & $\lceil \sqrt{3q}+\frac{1}{2}  \rceil$ & $\lceil \sqrt{\left( 2q+\frac{1}{4}\right) }+\frac{3}{2}\rceil$ & $t(2,q)$ \\ \hline \hline
13&7&7&7&8    \\ \hline
16&7& -- & {\bf 8} &9   \\ \hline
17&7&{\bf 8}& {\bf 8} &10     \\ \hline
19&8&{\bf 9}&8&10      \\ \hline
23&8&{\bf 9}& {\bf 9} &10     \\ \hline
25&9&{\bf 10}&9& 12    \\ \hline
27&9&--&9&12     \\ \hline
29&9&{\bf 10}& {\bf 10} &13     \\ \hline
31&9&{\bf 11}&10&14    \\ \hline
\end{tabular}
\medskip
\caption{{\small Lower bounds for arcs in $\P^2(\F_q)$ for $13\leq q\leq 31$.} }
\label{arco:generalizadoinferior-bis}
\end{center}
\end{table}
}

\begin{obs}\label{rem}
As shown in Table \ref{arco:generalizadoinferior}, 
the lower bound $\left\lceil \sqrt{\left( 2q+\frac{1}{4}\right) }+\frac{3}{2}\right\rceil$ is
better than $\lfloor \sqrt{2q} +2 \rfloor$ for $q=4,7,11$ and it is sharp for $q=2,3,7,8,9,11$. More generally, we have
$$\left\lceil \sqrt{\left( 2q+\frac{1}{4}\right) }+\frac{3}{2}\right\rceil = \lfloor \sqrt{2q} +2 \rfloor +1 > \lfloor \sqrt{2q} +2 \rfloor$$
for any $q$ such that $\frac{(h-2)(h-1)}{2}<q<\frac{(h-1)^2}{2}$ for some suitable integer $h\geq 4.$ 
Moreover, the lower bound of Proposition \ref{thm} is also
better than that obtained in \cite[Lemma 2.5]{FMMP} for $\mathbb{P}^2(\mathbb{F}_q)$.
\end{obs}

\medskip

\noindent Coming back to generalized $k$-arcs, let us note also that if $\mathcal{K}$ is a subset of $\P^2(\F_q)$, then

\medskip

\noindent $(\diamond)$ \quad $\mathcal{K}$ is a Veronesian $k$-arc $\iff$ $\mathcal{K}$ is a generalized $k$-arc such that all the conics passing through any $5$ of its points are irreducible.

\bigskip

If $q=2$, then we have $t_g(2,2)=m_g(2,2)=7$ and $t_v(2,2)=t(2,2)=4$, because any conic has at most $5$ points.

\medskip

Let $q=3$. Using Programs \ref{prog1} and \ref{prog2}, we obtain that $t_g(2,3)=m_g(2,3)=7$ and $t_v(2,3)=t(2,3)=4$, where a minimal complete Veronesian $4$-arc is given by $\left\{\ [1:1:1], [0:0:1], [0:1:0], [1:0:0]\ \right\}$.

\medskip

Now, let $q=4$ and consider a minimal generalized $k$-arc $\mathcal{K}_g\subset \P^2(\F_4)$.
Since $\mathcal{K}_g$ is the smallest complete generalized $k$-arc and any irreducible
conic in $\P^2(\F_4)$ has $q+1=5$ points, we deduce that $\mathcal{K}_g$
is not a $k$-arc. Therefore there are three collinear points in $\mathcal{K}_g$,
say $p_1,p_2,p_3$. Let $L$ be the line through these three points. 
Since $\mathcal{K}_g$ is a complete generalized $k$-arc,
we deduce that any point $p\notin\mathcal{K}_g$ is contained in a reducible
conic of type $L\cup l$, where $l$ is a secant line of $\mathcal{K}_g\setminus 
\{p_1,p_2,p_3\}$. This implies that 
$$2+2{k-3\choose 2}+k
\geq q^2+q+1=21\ .$$
Hence $k\geq 7$  and since $k\leq m_g(2,4)=7$ by Proposition \ref{prop4},
we conclude that $t_g(2,4)=m_g(2,4)=7$. By Program \ref{prog1}, we obtain that a minimal generalized $7$-arc is given by 
$$\{ [1:0:0], [0:1:0], [0:0:1], [1:1:1], [\alpha:\alpha^2:1],[\alpha^2:\alpha:1],[1:0:1] \}\ ,$$
where $\mathbb{F}_4^*=\langle\alpha\rangle$.
Moreover, by Program \ref{prog2}, one can see that $t_v(2,4)=t(2,4)=6$ and a minimal complete Veronesian $6$-arc is given by $$\{ [1:1:1], [0:0:1], [0:1:0], [1:0:0], [\alpha : \alpha^2 : 1], [\alpha^2 : \alpha : 1] \}\ .$$

\medskip

For $q\geq 5$, as to $t_v(2,q)$, we get the following lower bound.

\begin{teo}\label{teo2}
For $q\geq 5$, we have
$t_v(2,q)
\geq \lceil t_0 \rceil$, where $t_0$ is the smallest real positive solution
of the following inequality:
\[ {t_0\choose 5}(q-4)+{t_0\choose 2}(q-1)+ t_0 \geq q^2+q+1\ . \]
\end{teo}

\begin{proof}
Let $\mathcal{K}_v\subseteq \P^2(\F_q)$ be a complete Veronesian $k$-arc. Since 
any irreducible conic of $\P^2(\F_q)$ has $q+1$ points, from $(\diamond)$ we deduce that the maximal number of points covered by conics passing through $5$ of $k$ points of $\mathcal{K}_v$ is given by
\[ {k\choose 5}(q+1-5)+k={k\choose 5}(q-4)+k\ . \]
Moreover, the maximal number of points covered by lines passing through $2$ of 
$k$ points of $\mathcal{K}_v$ is
\[ {k\choose 2}(q+1-2)+k={k\choose 2}(q-1)+k\ . \]
Since the $k$ points of $\mathcal{K}_v$ are common to all the above conics and lines, to cover the projective plane we get the following inequality
\[ {k\choose 5}(q-4)+{k\choose 2}(q-1)+ k
\geq q^2+q+1\ . \] 
\end{proof}

\begin{obs}
In Table \ref{arco:generalizadoinferior1}, 
we compare the values of
$t(2,q)$ for complete $k$-arcs with the values of $t_v(2,q)$ obtained by Program \ref{prog2} and the lower bounds $\lceil t_0 \rceil$ of Theorem \ref{teo2} for complete Veronesian $k$-arcs. 
We would like to stress the fact that a Veronesian $k$-arc is a particular case of a generalized $k$-arc and that
if a set is complete as a Veronesian $k$-arc, could be not complete as a generalized $k$-arc.
Furthermore, by using Program \ref{prog2}, we found examples which
show that the lower bound $\lceil t_0 \rceil$ is sharp for $q=5,8,9,11$.
\end{obs}

{\tiny
\begin{table}[!h]
\begin{center}\renewcommand{\arraystretch}{1.6}
\setlength{\tabcolsep}{8pt}
\begin{tabular}{ | c | c  c  c | p{8.5cm} | }\hline
$~~q~~$ & $\lceil t_0 \rceil$  & $t_v(2,q)$ & $t(2,q)$ & Examples of complete Veronesian $k$-arcs \\ \hline \hline
5& \textbf{5} & 5 & 6 &  $[1:0:0], [0:1:0], [0:0:1], [1:1:1],[3:4:1]$ \\ \hline
7&5& 6 & 6& $[0:0:1],[0:1:0],[1:0:0],[1:1:1],[3:2:1],[4,3,1]$\\ \hline
8 & \textbf{6} & 6 & 6 & $[0:0:1],[0:1:0],[1:0:0],[1:1:1],[\alpha^3:\alpha^2:1], [\alpha^6:\alpha^4:1]$   \\ \hline 
9 & \textbf{6} & 6 & 6 & $[0:0:1],[0:1:0],[1:0:0],[1:1:1],[\alpha^6:\alpha^7:1], [\alpha^7:\alpha^3:1]$   \\ \hline
11 & \textbf{6} & 6 & 7 &
$[0:0:1],[0:1:0],[0:0:1],[1:1:1],[4:3:1],
[5:9:1]$    \\ \hline
\end{tabular}
\medskip
\caption{Minimal Veronesian $k$-arcs in $\P^2(\F_q)$ for $5\leq q\leq 11$.}
\label{arco:generalizadoinferior1}
\end{center}
\end{table}}

\smallskip

In a similar way as in the proof of Theorem \ref{teo2}, we get the following result.

\begin{teo}\label{t:2}
Let $\mathcal{K}_g\subseteq \P^2(\F_q)$ be a generalized $k$-arc with $q\geq 5$. If $\mathcal{K}_g$ is complete as a generalized arc, then 
\[ k\geq \begin{cases}
\lceil t_1 \rceil &\quad \text{if $\mathcal{K}_g$ is also a $k$-arc}\\
\left\lceil t_2 \right\rceil & \quad \text{if $\mathcal{K}_g$ is not a $k$-arc}\ ,\\
\end{cases}
\]
where $t_1$ and $t_2$ are the minimum real positive solutions of the following
inequalities: 
\[ {t_1\choose 5}(q-4)
+ t_1 \geq q^2+q+1\ ,\ \ \ \ \text{and}\ \]
\[ {t_2-3\choose 5 }(q-4)+{t_2-2\choose 4}(6q-12)+{t_2-3\choose 2}(q-2)+t_2\geq q^2+3\ . \]

\smallskip

\noindent In particular, we have $t_g(2,q)\geq\min\left\{ \lceil t_1 \rceil, \lceil t_2 \rceil \right\}$ . 
\end{teo}

\begin{proof}
Let $\mathcal{K}_g\subseteq \P^2(\F_q)$ be a complete generalized $k$-arc.
If $\mathcal{K}_g$ is also a $k$-arc, then $\mathcal{K}_g$ is a Veronesian 
$k$-arc complete by (irreducible) conics. Thus, in this situation, by arguing in a similar 
way as in Theorem \ref{teo2},
we get $k\geq\lceil t_1 \rceil$, where ${t_1\choose 5}(q-4)
+ t_1 \geq q^2+q+1$.

Suppose now that $\mathcal{K}_g$ is not a $k$-arc.
From Lemma \ref{prop321} it follows that we can write $\mathcal{K}_g=\mathcal{K}\cup\mathcal{P}$, where $\mathcal{K}$ is a (Veronesian) $(k-3)$-arc and $\mathcal{P}$ are $3$ distinct collinear points.
The conics $\gamma_0$ passing through any $5$ distinct points of $\mathcal{K}$ are irreducible because $\mathcal{K}$ is a $(k-3)$-arc and they cover at most
\begin{equation} \label{des:1}
{3\choose 0}{k-3\choose 5 }(q+1-5)+k 
\end{equation}
points of the projective plane. On the other hand, the conics $\gamma_1$ passing through any $4$ distinct points of $\mathcal{K}$ 
and one point of $\mathcal{P}$, can be either reducible or irreducible. Since the reducible conics cover
$2q+1>q+1$ points of the plane, we can assume that all the $\gamma_1$'s are reducible.
Thus they cover at most
\begin{equation}\label{des:2}
 {3\choose 1}{k-3\choose 4 }(2q+1-5)+k 
\end{equation}
points of the plane. Similarly, the conics $\gamma_2$ passing through any $3$ distinct points of $\mathcal{K}$ 
and $2$ distinct points of $\mathcal{P}$, can be either reducible or irreducible. As in the previous case,
we can assume that all the $\gamma_2$'s are reducible conics. Therefore, they cover at most 
\begin{equation}\label{des:3}
 {3\choose 2}{k-3\choose 3 }(2q+1-5)+k 
 \end{equation}
points of $\mathbb{P}^2(\mathbb{F}_q)$. Finally, the conics $\gamma_3$ passing through any $2$ distinct points 
of $\mathcal{K}$ and the $3$ points of $\mathcal{P}$ are all reducible and they cover at most
\begin{equation}\label{des:4}
 {3\choose 3}{k-3\choose 2 }(q+1-3)+(q+1-3)+k 
 \end{equation}
points of the projective plane. Since the $k$ points of $\mathcal{K}_g$ are common to all the above conics $\gamma_i$ for $i=0,1,2,3$, using the expressions from (\ref{des:1}) to (\ref{des:4}), we get  
\[ S(k):={k-3\choose 5 }(q-4)+{k-2\choose 4}(6q-12)+\left[ {k-3\choose 2}+1\right](q-2)+k, \]
which is greater than or equal to the number of points covered by the span of all the conics passing through any $5$ distinct points of
$\mathcal{K}_g$. So, if $t_2$ is the minimum real positive solution of $S(t_2)\geq q^2+q+1$, then we have $k\geq \lceil t_2 \rceil$.
\end{proof}

\bigskip

\begin{obs}
In Tables \ref{arco:generalizadoinferior2} and
\ref{arco:generalizadoinferior1-bis}, 
we compare the two lower bounds of 
Theorem \ref{t:2} for a generalized $k$-arc
in $\P^2(\F_q)$.  Moreover, for $5\leq q\leq 11$, in Table \ref{arco:generalizadoinferior2} we find the exact values of $t_g(2,q)$ by using Program \ref{prog1}, and we give their respective examples showing that the lower bound $\lceil t_2\rceil$ is sharp for $q=7,8,9,11$. 
\end{obs}

{\tiny
\begin{table}[!h]
\begin{center}\renewcommand{\arraystretch}{1.6}
\setlength{\tabcolsep}{8pt}
\begin{tabular}{ | c | c c  c  | p{5.2cm} | }\hline
$~~q~~$ & $\lceil t_1 \rceil$ & $\lceil t_2  \rceil$ & $t_g(2,q)$ &  Examples of complete generalized $k$-arcs \\ \hline \hline
\multirow{2}*{5}&\multirow{2}*{8}&\multirow{2}*{6}&\multirow{2}*{7}&  $[1:0:0], [0:1:0], [0:0:1], [1:1:1],$ \\
&&&& $ [1:0:1], [1:1:0], [0:1:1]$  \\ \hline
\multirow{2}*{7}&\multirow{2}*{7}&\multirow{2}*{7}&\multirow{2}*{\textbf{7}}& $[0:0:1],[0:1:0],[1:0:0],[1:1:1],$\\
&&&& $[3:0:1],[5:5:1],[2:4:1]$   \\ \hline
\multirow{2}*{8}&\multirow{2}*{7}&\multirow{2}*{7}&\multirow{2}*{\textbf{7}}&  $[0:0:1],[0:1:0],[1:0:0],[1:1:1],$ \\
&&&& $[\alpha:\alpha^2:1],[\alpha^6:1:0],[\alpha^3:\alpha^6:1]$   \\ \hline
\multirow{2}*{9}&\multirow{2}*{7}&\multirow{2}*{7}&\multirow{2}*{\textbf{7}}& $[0:0:1],[0:1:0],[1:0:0],[1:1:1],$ \\
&&&& $[\alpha^3:\alpha^7:1],[1:\alpha^7:1],[\alpha^6:\alpha^3:1]$   \\ \hline
\multirow{2}*{11}&\multirow{2}*{7}&\multirow{2}*{7}&\multirow{2}*{\textbf{7}}& $[0:0:1],[0:1:0],[0:0:1],[1:1:1]$ \\
&&&& $[4:5:1],[1:5:1],[5:3:1]$    \\ \hline
\end{tabular}
\medskip
\caption{Minimal generalized $k$-arcs in $\P^2(\F_q)$ for $5\leq q\leq 11$.}
\label{arco:generalizadoinferior2}
\end{center}
\end{table}

\begin{table}[!h]
\begin{center}\renewcommand{\arraystretch}{1.6}
\setlength{\tabcolsep}{8pt}
\begin{tabular}{ | c | c c c c c c c c c | }\hline
$~~q~~$ & 13 & 16 & 17 & 19 & 23 & 25 & 27 & 29 & 31  \\ \hline \hline
$\lceil t_1 \rceil$ & 7 & 8 & 8 & 8 & 8 & 8 & 8 & 8 & 8 \\ \hline
$\lceil t_2 \rceil$ & 7 & 7 & 7 & 7 & 7 & 7 & 7 & 7 & 7 \\ \hline
\end{tabular}
\medskip
\caption{Lower bounds for complete generalized $k$-arcs in $\P^2(\F_q)$ for $13\leq q\leq 31$.}
\label{arco:generalizadoinferior1-bis}
\end{center}
\end{table}
}

\bigskip

Finally, looking at the number of $3$-secant lines of a generalized $k$-arc, we can also obtain the following result.

\medskip

\begin{prop}\label{teo3}
Let $\mathcal{K}_g\subseteq \P^2(\F_q)$ be a complete generalized $k$-arc with $q\geq 5$.
Let $T\in\mathbb{Z}_{\geq 0}$ be the number of distinct $3$-secant lines of $\mathcal{K}_g$.
Then $k\geq \lceil t_3 \rceil$, where $t_3$ is the minimum real positive solution of the following
inequality
\[
{t_3\choose 5}(q-4)+\frac{T}{2}\Big[(t_3-3)(t_3-4)+1-T\Big]q+t_3\geq q^2+q+1\ . \]
\end{prop}
 
\begin{proof}
Write $\mathcal{K}_g:=\{p_1, \dots , p_k\}$ and let $I:=\{\ (p,\gamma)\ | \ p\in\mathbb{P}^2(\mathbb{F}_q)\ ,\ \gamma\in\Gamma\ ,\ p\in\gamma\}\ ,$
where $\Gamma$ is the set of all the conics in $\mathbb{P}^2(\mathbb{F}_q)$ passing through any $5$ distinct points of $\mathcal{K}_g$. Since any reducible conic of $\Gamma$ contains at least a $3$-secant line, we deduce that the number of distinct reducible conics in $\Gamma$ is $\delta$ given by
$$\delta:={T\choose 2}+T\left[{k-3\choose 2}-(T-1)\right]\ ,$$
while the number of distinct irreducible conics in $\Gamma$ is ${k\choose 5}-\delta$, because $|\Gamma|={k\choose 5}$. So, the cardinality of $I$ is
$$\left[{k\choose 5}-\delta\right](q+1)+\delta(2q+1)={k\choose 5}(q+1)+\delta q\ ,$$
that is, $$|I|={k\choose 5}(q+1)+\left\{{T\choose 2}+T\left[{k-3\choose 2}-(T-1)\right]\right\}q\ .$$
Now, observe that the number of distinct conics in $\Gamma$ passing through a fixed point $p_j\in\mathcal{K}_g$ is given by ${k-1\choose 4}$. Note that $I$ can be written as a union $I_1\cup I_2$, where $I_1:=\left\{ (p_j,\gamma)\in I\ |\ p_j\in\mathcal{K}_g\right\}$ and $I_2:=\left\{ (q,\gamma)\in I\ |\ q\notin\mathcal{K}_g\right\}$. Since $|I_1|=k{k-1\choose 4}$ and $|\mathbb{P}^2(\mathbb{F}_q)|-k\leq |I_2|$, we conclude that $$|I|-k\left[{k-1\choose 4}-1\right]\geq |\mathbb{P}^2(\mathbb{F}_q)|\ ,$$ i.e. $|I|-5{k\choose 5}+k\geq q^2+q+1$, which leads to the inequality of the statement. 
\end{proof}

\medskip

\section*{Appendix: Magma programs}\label{Appendix}

\smallskip

Let us give here the two main Magma \cite{Magma} programs we used to construct 
the tables and the examples of the previous sections.

\bigskip

\begin{prog}\label{prog1} Function to find all the complete generalized $k$-arcs in $\mathbb{P}^2(\mathbb{F}_q)$.

\medskip

\begin{verbatim}
function GArc(q,k)
P<[x]>:=ProjectivePlane(GF(q));
pts:={P![1,0,0],P![0,1,0],P![0,0,1],P![1,1,1]};
S1:={@ e : e in Subsets({p : p in Points(P)} diff pts,k-4) @};
N:=0;
 for j in [1..#S1] do
 ptss:=pts join S1[j];
 pplane:={};
 T:={@ C : C in Set(&cat[[f : f in Basis(&meet[Ideal(Cluster(p)):
 p in S]) | Degree(f) eq 2] : S in Subsets(ptss,5)]) @};
  if #T ge 1 then
  g:=0;
  repeat
  g:=g+1;
  pplane:=pplane join { b : b in Points(Curve(P,[T[g]])) };
  pc:={ b : b in Points(Curve(P,[T[g]])) };
  card:=#(pc meet ptss);
  until g eq #T or card ge 6;
   if card le 5 and #pplane eq q^2+q+1 then
   N:=N+1;
   "Set",N,"=",ptss;
   end if; 
  end if;
 end for;
return "end";
end function;
\end{verbatim}
\end{prog}

\newpage

\begin{prog}\label{prog2} Function to find all the complete Veronesian $k$-arcs in $\mathbb{P}^2(\mathbb{F}_q)$.

\medskip

\begin{verbatim}
function VArc(q,k)
P<[x]>:=ProjectivePlane(GF(q));
pts:={P![1,0,0],P![0,1,0],P![0,0,1],P![1,1,1]};
S1:={@ e : e in Subsets({p : p in Points(P)} diff pts,k-4) @};
N:=0;
 for j in [1..#S1] do
 ptss := pts join S1[j];
 pplane:={};
 T1:={@ C : C in Set(&cat[[f : f in Basis(&meet[Ideal(Cluster(p))
 : p in S]) | Degree(f) eq 1] : S in Subsets(ptss,2)]) @};
 T2:={@ C : C in Set(&cat[[f : f in Basis(&meet[Ideal(Cluster(p))
 : p in S]) | Degree(f) eq 2] : S in Subsets(ptss,5)]) @};
 g:=0;
 repeat
 g:=g+1;
 pplane:=pplane join { b : b in Points(Curve(P,[T1[g]])) };
 pt:={ b : b in Points(Curve(P,[T1[g]])) };
 card1:=#(pt meet ptss);
 until g eq #T1 or card1 ge 3;
  if card1 le 2 then
   if #T2 ge 1 then
   g:=0;
   repeat
   g:=g+1;
   pplane:=pplane join { b : b in Points(Curve(P,[T2[g]])) };
   pc:={ b : b in Points(Curve(P,[T2[g]])) };
   card2:=#(pc meet ptss);
   until g eq #T2 or card2 ge 6;
   else
   card2:=0;
   end if;
   if #pplane eq q^2+q+1 and card2 le 5 then
   N:=N+1;
   "Set",N,"=",ptss;
   end if;
  end if;
 end for;
return "end";
end function;
\end{verbatim}
\end{prog}

\end{document}